\documentclass[11pt]{article}
\usepackage{url,epsf,amsmath,amssymb,amsfonts,graphicx}
\usepackage{amsmath,amsfonts,amssymb}
\usepackage[]{color}
\usepackage{epsfig}
\usepackage{amsmath}
\usepackage{hyperref}
\usepackage{float}
\usepackage{graphicx}
\usepackage{amsfonts}
\usepackage{enumerate}

\newtheorem{theorem}{\bf Theorem}[section]
\newtheorem{corollary}[theorem]{\bf Corollary}

\newtheorem{lemma}[theorem]{\bf Lemma}

\newtheorem{conjecture}[theorem]{\bf Conjecture}

\begin{document}

\begin{center}
\textbf{\Large A note on the Thue chromatic number \\ \smallskip of lexicographic products of graphs}
\end{center}

\bigskip

\begin{center}
{\Large Iztok Peterin}\footnote{%
Institute of Mathematics and Phisics, Faculty of Electrical Engineering and
Computer Science, University of Maribor, Maribor, Slovenia
[iztok.peterin@uni-mb.si]}, {\Large {Jens Schreyer}}\footnote{%
Institute of Mathematics, Faculty of Mathematics and Natural Sciences,
Technical University Ilmenau, Ilmenau, Germany [jens.schreyer@tu-ilmenau.de]}%
,

\smallskip

{\Large Erika \v Skrabul\kern-0.035cm\char39\kern-0.03cm \'akov\'a}\footnote{%
Institute of Control and Informatization of Production Processes, Faculty of
Mining, Ecology, Process Control and Geotechnology, Technical University of
Ko\v sice, Ko\v sice, Slovakia [erika.skrabulakova@tuke.sk]} and {\Large %
Andrej Taranenko}\footnote{%
Department of Mathematics and Computer science, Faculty of Natural Sciences
and Mathematics, University of Maribor, Maribor, Slovenia
[andrej.taranenko@uni-mb.si]}

\bigskip
\end{center}

{\abstract
A sequence is called {\em non-repetitive} if no of its subsequences forms a re\-pe\-ti\-ti\-on (a sequence $r_1,r_2,\dots,r_{2n}$ such that $r_i=r_{n+i}$ for all $1\leq i \leq n$). Let $G$ be a graph whose vertices are coloured. A colouring  $\varphi$ of the graph $G$ is {\em non-repetitive} if the sequence of colours on every path in $G$ is non-repetitive.  The {\em Thue chromatic number}, denoted by $\pi (G)$, is the minimum number of colours of a non-repetitive colouring of $G$. \\
In this short note we present a general upper bound for the Thue chromatic number for the lexicographic product
$G\circ H$ of graphs $G$ and $H$ with respect to some properties of the factors.
This upper bound is then used to derive the exact values for $\pi(G\circ H)$ when
$G$ is a complete multipartite graph and $H$  is an arbitrary graph.
}\medskip

\noindent \textbf{Keywords:} non-repetitive colouring, Thue chromatic
number, lexicographic product of graphs\medskip

\noindent \textbf{AMS subject classification (2010)}: 05C15, 05C76

\section{Introduction and preliminaries}


In 1906 the Norwegian mathematician and number theoretician Axel Thue started
the systematic study of word structure. Thue \cite{Th06} has shown that
there are arbitrarily long non-repetitive sequences over three symbols,
where a sequence $a_{1}a_{2}\ldots $ is called \emph{non-repetitive} if it
does not contain a subsequence of consecutive elements, the first half of which is
exactly the same as its second half. A sequence $r_{1}\ldots r_{2n}$ such that $%
r_{i}=r_{n+i}$ for all $1\leq i\leq n$ is called a \emph{repetition}.

With the development of computer-science the research on string-type chains
became more and more popular. Non-repetitive sequences found their
applications beside mathematics or informatics in a lot of very different
areas counting from information security management up to music.

By Alon, Grytczuk, Ha\l uszcak and Riordan \cite{AGHR02} non-repetitive
sequences were introduced also to graph theory: Let $G$ be a simple
graph and let $\varphi $ be a proper colouring of its vertices, $\varphi
:V(G)\rightarrow \{1,\ldots ,k\}$. We say that $\varphi $ is \emph{%
non-repetitive} if for any simple path on vertices $v_{1}\ldots v_{2n}$ in $%
G $ the associated sequence of colours $\varphi (v_{1})\ldots \varphi
(v_{2n})$ is not a repetition. The minimum number of colours in a
non-repetitive colouring of a graph G is the \emph{Thue chromatic number} $%
\pi (G)$. For the case of list-colourings let the \emph{Thue choice number} $%
\pi _{ch}(G)$ of a graph $G$ denotes the smallest integer $k$ such that for
every list assignment $L:V(G)\rightarrow 2\sp{\mathbb{N}}$ with minimum list
length at least $k$, there is a colouring of vertices of $G$ from the
assigned lists such that the sequence of vertex colours of no path in $G$
forms a repetition. If a graph $G$ is non-repetitively list colourable for
every list assignment $L$ with list size at least $k$, we call $G$ \emph{%
non-repetitively k-choosable}. Hence, $\pi _{ch}(G)$ is the smallest integer
$k$ such that $G$ is non-repetitively $k$-choosable.

In \cite{Gr07} various questions concerning non-repetitive colourings of
graphs have been formulated. We deal with the problem to find the minimum
number of colours that can be used to colour all vertices of an arbitrary graph
such that the obtained colouring is non-repetitive. The problem of
determining the Thue chromatic number of a graph was studied among others in \cite%
{AGHR02,BGKNP07,CzGr07,KPZ11}.

The \emph{lexicographic product} or graph composition $G\circ H$ \emph{of
graphs $G$ and $H$} is a graph such that the vertex set of $G\circ H$ is the
Cartesian product $V(G)\times V(H)$ and any two vertices $(u,v)$ and $(x,y)$
are adjacent in $G\circ H$ if either $u$ is adjacent with $x$ in
$G$ or $u=x$ and $v$ is adjacent with $y$ in $H$. For any vertex $v\in V(G)$
we call the set $\{(v,w):w\in V(H)\}$ an $H$-\textit{layer} (through $v$) and will
be denoted by $v[H]$. A subgraph induced by $v[H]$ is clearly isomorphic to $H$.
If $G\circ H$ is a coloured graph, then we say that an $H$-layer $v[H]$ is
\textit{rainbow coloured} whenever all vertices of $v[H]$ have pairwise different
colours.
Note that the lexicographic product is in general non-commutative: $%
G\circ H\neq H\circ G$. The independence number of a lexicographic product
may be easily calculated from that of its factors (see \cite{GeSt75}): $%
\alpha (G\circ H)=\alpha (G)\alpha (H)$ and the clique number of a
lexicographic product is multiplicative as well: $\omega (G\circ H)=\omega
(G)\omega (H)$.

The Thue chromatic number of $G\circ H$ when $G$ is a path and $H$ is either
an empty graph $E_k$ or
a complete graph $K_k$ (also called the blow-up of G by H) was studied in \cite%
{KPZ11}. Here we give some upper bounds for the Thue chromatic number of
the lexicographic product of arbitrary graphs and demonstrate the
tightness of the bounds by some examples. As a side result we show that there exist families
of graphs where the Thue chromatic number and the Thue choice number are the
same\footnote{%
Remark that in general the Thue chromatic number and the Thue choice number
of the same graph may have arbitrary large difference (see \cite{FOOMZ11}),
however the most interesting open problem from this area is whether the Thue
chromatic number of a path equals its Thue choice number (see \cite{GPZ10}).}.

An easy observation about non-repetitive sequences is the following: If a
non-repetitive sequence is interrupted by non-repetitive sequences using a
distinct set of symbols, then the resulting new sequence remains
non-repetitive. Formally we get the following lemma, proved in \cite{HJSS09}%
, where for a sequence of symbols $S=a_{1}\ldots a_{n}$ with $a_{i}\in
\mathbb{A}$, for all $1\leq k\leq \ell \leq n$, the block $%
a_{k}a_{k+1}\ldots a_{\ell }$ is denoted by $S_{k,\ell }$.\bigskip

\begin{lemma}
\textbf{(Havet at al.)}\label{1} Let $A=a_{1}\ldots a_{m}$ be a
non-repetitive sequence with $a_{i}\in \mathbb{A}$ for every $i\in
\{1,\ldots ,m\}$. Let $B^{i}=b_{1}^{i}\ldots b_{m_{i}}^{i}$; $0\leq i\leq r+1
$, be non-repetitive sequences with $b\sp{i}_{j}\in \mathbb{B}$ for every $%
i\in \{0,\ldots ,r+1\}$ and $j\in \{1,\ldots ,m_{i}\}$. If $\mathbb{A}\cap
\mathbb{B}=\emptyset $ then $S=B^{0}A_{1,n_{1}}B^{1}A_{n_{1}+1,n_{2}}\ldots
B^{r}A_{n_{r}+1,m}B^{r+1}$ is a non-repetitive sequence.
\end{lemma}

A \emph{rainbow} sequence, i.e. a sequence of pairwise different elements, is trivially non-repetitive. This implies the following corollary:

\begin{corollary}
\label{rainbow} Let $A=a_{1}\ldots a_{m}$ be a rainbow sequence with $%
a_{i}\in \mathbb{A}$ for all $i\in \{1,\ldots ,m\}$. For $i\in \{0,\ldots
,r+1\}$ let $b_{i}\notin \mathbb{A}$. Then $S=b_{0}A_{1,n_{1}}b_{1}A_{n_{1}+1,n_{2}}%
\ldots b_{r}A_{n_{r}+1,m}b_{r+1}$ is a non-repetitive sequence.
\end{corollary}


\section{Main results}
We start with a general upper bound for the Thue chromatic number of
lexicographic products. Recall that $\alpha (G)$ is the notation for the
independence number of a graph $G$.

\begin{theorem}
\label{up} For all simple graphs $G$ and $H$ we have that
\begin{equation*}
\pi (G\circ H)\leq \pi (H)+(|V(G)|-\alpha (G))|V(H)|.
\end{equation*}
\end{theorem}

\begin{proof}
Let $M$ be an independent set of vertices in $G$ of cardinality $\alpha (G)$.
Colour all $H$-layers corresponding to the vertices from $M$ with the same
set of $\pi(H)$ colours, so that the copy of a graph $H$ in each $H$-layer is
coloured non-repetitively and for every two vertices $w',w\in M$ the colouring
of $w'[H]$ is the same as the colouring of $w[H]$. For all other vertices use different
colours. Obviously such a colouring uses $\pi(H)+(|V(G)|-\alpha (G))|V(H)|$
colours. We claim that the obtained colouring, say $\varphi$, is a non-repetitive
colouring of $G\circ H$.

Assume that there exists a repetitive path
$P=v_1\ldots v_rv_{r+1} \ldots v_{2r}$ in $G \circ H$, such that
$\varphi(v_1)=\varphi(v_{r+1}),\ldots ,\varphi(v_{r})=\varphi(v_{2r})$.
Note that a vertex from $(V(G)-M)\times V(H)$ is not on $P$, since it has a
unique colour. Thus for each $j\in\{1,\ldots, 2r\}$ is $v_j \in w[H]$ for some
$w\in M$. Since every
$H$ layer is coloured non-repetitively, not all vertices of $P$ can be in the same $H$ layer. Without loss of generality suppose that
$v_1 \in w_1[H]$  and $v_2 \in w_2[H]$. As $e=v_1v_2$ is an edge in $P$,
there exists an edge $e'=w_1w_2$ in $G$, a contradiction with $M$ being an
independent set of vertices. Hence, our assumption is wrong and  $\varphi$ is non-repetitive. \hfill $\Box $\bigskip
\end{proof}

Before showing the sharpness of the bound for $\pi(G \circ H)$ in
Theorem~\ref{up} we prove a result, that is a little stronger than needed, as it is dealing with vertex list colourings.

\begin{theorem}
\label{pi} If $G$ is a a graph on $n$ vertices, then the following
statements hold:

\begin{enumerate}
\item $\pi (G)\leq \pi _{ch}(G)\leq n-\alpha (G)+1$.

\item if $G$ is a complete multipartite graph, then $\pi (G)=\pi
_{ch}(G)=n-\alpha (G)+1$.
\end{enumerate}
\end{theorem}

\begin{proof}~~
\begin{enumerate}
\item
Let $G$ be a graph on $n$ vertices. As every non-repetitive $k$-colouring of $G$ can
be considered as a non-repetitive list-colouring of $G$ from identical lists of size $k$,
the first inequality ($\pi(G)\leq\pi_{ch}(G)$) trivially holds.\\
To show the second inequality let $M$ be a maximum independent set of vertices from $V(G)$
with $|M|=\alpha(G)$, and let $L:V(G)\to 2\sp{\mathbb{N}}$ be any list assignment such that
each list length is at least $n-\alpha(G)+1$. Colour the vertices belonging to $V(G)- M$
with pairwise different colours from their lists, and remove all colours used by any of
these vertices from the lists of the vertices of $M$. As each list is of length at least
$n-\alpha(G)+1$, at least one colour from the list of each vertex $x\in M$  remains, and this
will be used to colour the vertex $x$. Now consider the sequence of vertex colours of any
path in $G$. The subsequence of colours belonging to the vertices of $V(G)- M$
constitute a rainbow sequence, which is interrupted by single colours belonging to the
vertices of $M$. Hence, by Corollary \ref{rainbow} such a colouring is non-repetitive,
which proves that $\pi(G)\leq\pi_{ch}(G)\leq n-\alpha(G)+1$.

\item
Let $G$ be a complete multipartite graph of order $n$ with partite sets $V_1,...,V_m$.
To prove the statement it is sufficient to show that $\pi(G)\ge n-\alpha(G)+1$ which we
will prove by contradiction. Assume there is a non-repetitive
($n-\alpha(G)$)-colouring $\varphi$ of $G$. Because there are $n$ vertices coloured by
$n-\alpha(G)$ different colours, by pigeon-hole principle the set
$M:=\{x\in V(G):\exists \ x'\in V(G)-\{x\}:\varphi(x)=\varphi(x')\}$ of vertices
without unique colour consists of at least $\alpha(G)+1$ vertices. Because all partite sets
consist of at most $\alpha(G)$ vertices there exist two vertices $x$ and $y$ of $M$
belonging to different partite sets. Without loss of generality we assume $x\in V_1$
and $y\in V_2$. Since adjacent vertices must receive different colours the colour
$\varphi(x)$ can only appear in $V_1$ and the colour $\varphi(y)$ can only appear in
$V_2$. Hence, there must be a vertex $x'\in V_1$ with $\varphi(x')=\varphi(x)$ and a vertex
$y'\in V_2$ with $\varphi(y')=\varphi(y)$. But then the colour sequence of the path
$P=(x,y,x',y')$ is  repetitive, a contradiction. \hfill $\Box $\bigskip
\end{enumerate}
\end{proof}

As an immediate corollary we obtain several infinite subclasses of complete
multipartite graphs, where the graph parameters $\pi$ and $\pi_{ch}$
coincide.

\begin{corollary}\label{example}The following statements hold:
\begin{itemize}
\item $\pi(K_n)=\pi_{ch}(K_n)=n$ for the complete graph $K_n$ on $n$ vertices.
\item $\pi(S_n)=\pi_{ch}(S_n)=2$ for a star $S_n$ on $n+1$ vertices.
\item $\pi(K_{m,n})=\pi_{ch}(K_{m,n})=\min\{m,n\}+1$ for a complete bipartite graph $K_{m,n}$.
\end{itemize}
\end{corollary}

For the lower bound we strongly suspect that the following is true.

\begin{conjecture}
\label{down} For all simple graphs $G$ and $H$ we have that
\begin{equation*}
\pi (H)+(\pi (G)-1)|V(H)|\leq \pi (G\circ H).
\end{equation*}
\end{conjecture}

This conjecture is true for $\pi(P_n\circ E_k)$ and for $\pi(P_n\circ K_k)$ by Theorem 1.2 and
Theorem 1.4, respectively, of \cite{KPZ11}. Moreover it is sharp for $\pi(P_n\circ E_k)$ for
$n\geq 4$ and $k>2$, but not for $\pi(P_n\circ K_k)$ for $n\geq 28$  by the same theorems. We show that
Conjecture \ref{down} holds also for the lexicographic product of the complete multipartite
graph $K_{n_1,\ldots,n_k}$ with any graph $H$. This represents another good reason to believe that
Conjecture \ref{down} is true: namely paths are very sparse graphs while complete multipartite
graphs represent dense graphs with respect to the number of edges.

\begin{theorem}
\label{downcomplete} For a complete multiparite graph $G$ and an arbitrary graph $H$ we have that
\begin{equation*}
\pi (H)+(\pi (G)-1)|V(H)|\leq \pi (G\circ H).
\end{equation*}\end{theorem}

\begin{proof}
Let $G$ be a complete multipartite graph $K_{n_1,\ldots,n_k}$ where $V_1,\ldots,V_k$ form a
partition of $V(G)$ with $|V_1|=n_1,\ldots,|V_k|=n_k$ and $n_1+\cdots +n_k=n$.
Towards a contradiction suppose that there exists a non-repetitive colouring
$\varphi$ of $G\circ H$ using less than $\pi(H)+(\pi(G)-1)|V(H)|$ colours. Since
$\pi(H)\leq |V(H)|$, we have less than $\pi(G)|V(H)|$ colours for $\varphi$ in general.
Hence there exists $g \in V(G)$ such that in the layer $g[H]$ we have less than $|V(H)|$ different
colours. We may assume that $g\in V_1$ and that $\varphi(g,h)=1=\varphi(g,h')$. If there
exists  $g'\in V(G)-V_1$ with two different vertices in $g'[H]$, say $(g',h_1)$ and $(g',h_2)$,
of the same colour, then we have a repetition on $(g,h)(g',h_1)(g,h')(g',h_2)$, which is
a contradiction. Therefore for every $g'\in V(G)-V_1$ the layer $g'[H]$ must be rainbow coloured.
Let now $g_1,g_2\in V(G)-V_1$, $g_1\neq g_2$. If there are vertices of the same colour in
$g_1[H]$ and in $g_2[H]$, say $(g_1,h_3)$ and $(g_2,h_4)$, then we have again a repetition
$(g,h)(g_1,h_3)(g,h')(g_2,h_4)$, which is not possible. Thus for all pairs of different vertices $g_1,g_2\in V(G)-V_1$ all colours in  $g_1[H]$ and $g_2[H]$ must be pairwise different. This means
we have used $|(V(G)-V_1)\times V(H)|$ colours on $(G-V_1)\circ H$. By Theorem \ref{pi}, $\pi (G)=n-{\rm max}\{n_1,\ldots,n_k\}+1$.
Considering the number of colours in $\varphi$ and the number of colours we have used for $(G-V_1)\circ H$, the number of available colours for the layer $g[H]$ is the less than
\begin{eqnarray*}
	\pi(H) + (\pi(G)-1)\cdot|V(H)| - |(V(G)-V_1)|\cdot|V(H)| & = \\
	\pi(H) + (n-\max\{n_1, \ldots, n_k\}+1-1)\cdot|V(H)| - (n-n_1)\cdot|V(H)| & = \\
	\pi(H) + (n-\max\{n_1, \ldots, n_k\}-n+n_1)\cdot|V(H)| & =\\
	\pi(H) + (n_1-\max\{n_1, \ldots, n_k\})\cdot|V(H)|&\leq &\pi(H).
\end{eqnarray*}
This yields a final contradiction since there are less than $\pi(H)$ colours left for $g[H]$, which results in a repetition in the
$H$-layer $g[H]$. \hfill $\Box $\bigskip
\end{proof}

The family of complete multipartite graphs is one that establishes the tightness of the bounds for $G\circ H$ given by Theorem~\ref%
{up} and Theorem~\ref{downcomplete}.

\begin{theorem}
\label{multi} Let $G$ be a complete multipartite graph. Then $%
\pi(H)+(\pi(G)-1)|V(H)|= \pi(G \circ H) = \pi(H)+(|V(G)|-\alpha (G))|V(H)|$
\end{theorem}

\begin{proof}
Theorem~\ref{pi} shows that $\pi(G)=n-\alpha(G)+1$, where $\alpha (G)$ is the independence
number of $G$ being a complete multipartite graph on $n$ vertices. The statement of
Theorem~\ref{multi} then directly follows from Theorem~\ref{up} and Theorem~\ref{downcomplete}.  \hfill $\Box $\bigskip
\end{proof}

Theorem~\ref{multi} and Corollary~\ref{example} then give the following Corollary.

\begin{corollary} For any graph $H$ we have that
\begin{itemize}
\item $\pi(K_n \circ H)=\pi(H)+(n-1)|V(H)|$, where $K_n$ is complete graph on $n$ vertices.

\item $\pi(S_n \circ H)=\pi(H)+|V(H)|$, where $S_n$ is a star on $n+1$
vertices.

\item $\pi(K_{m,n} \circ H)=\pi(H)+\min\{m,n\}\cdot |V(H)|$, where $K_{m,n}$ is a
complete bipartite graph on $m+n$ vertices.
\end{itemize}
\end{corollary}

\section{Acknowledgments}


This work was supported by the Slovak Research and Development Agency under the contract No. APVV-0482-11, by the grants VEGA 1/0130/12 and KEGA 040TUKE-4/2014. The Slovenian co-authors were supported by the research grant P1-0297 and are also with IMFM, Jadranska 19, 1000 Ljubljana, Slovenia.


\end{document}